\documentclass[11pt,reqno]{amsart}
\usepackage[active]{srcltx} % SRC Specials: DVI [Inverse] Search
\usepackage{latexsym}
\usepackage{amssymb,amsmath,amsfonts}
\usepackage[ansinew]{inputenc}
\usepackage{graphicx}
\usepackage{color}
\usepackage{url}
\usepackage[colorlinks]{hyperref}
\numberwithin{equation}{section}
\newtheorem{thm}{Theorem}[section]
\newtheorem{cor}[thm]{Corollary}
\newtheorem{lem}[thm]{Lemma}
\newtheorem{prop}[thm]{Proposition}
\newtheorem{ex}{Example}

\theoremstyle{definition}

\newtheorem{rem}{Remark}

\usepackage{fancyhdr}
\usepackage{graphicx}
\begin{document}

\title{Eigenvalue analysis of constrained minimization problem for homogeneous polynomial}%
 %\author{Yisheng Song$^{1,2}$,  Liqun Qi$^1$\\ 1.Department of Applied Mathematics, The Hong Kong Polytechnic University, Hung Hom, Kowloon, Hong Kong\\2. College of Mathematics and Information Science, Henan Normal University, XinXiang HeNan,  P.R. China, 453007.}
 \thanks{Email: songyisheng1@gmail.com(Song); maqilq@polyu.edu.hk(Qi).}
 \thanks{The work was supported by the Hong Kong Research Grant Council (Grant No. PolyU 501808, 501909, 502510, 502111) and the first author was supported partly by the National Natural Science Foundation of P.R. China (Grant No. 11071279, 11171094, 11271112)  and by the Research Projects of Science and Technology Department of Henan Province(Grant No. 122300410414).}
 \maketitle
 \begin{center}{Yisheng Song$^{1,2}$ and Liqun Qi$^1$}\\\vskip 2mm
 1. Department of Applied Mathematics, The Hong Kong Polytechnic University, Hung Hom, Kowloon, Hong Kong \\\vskip 2mm
 2. College of Mathematics and Information Science,
Henan Normal University, XinXiang HeNan,  P.R. China, 453007.\end{center}
 %
%\subjclass{47H06, 47J05, 47J25, 47H10, 47H17.}%In particular, the nonlinear Perron-Frobenius property is generalized to the nonnegative higher order tensor pairs $(\mathcal{A,B})$.
\vskip 4mm
\begin{quote}{\bf Abstract.}\
In this paper, the concepts of Pareto $H$-eigenvalue and Pareto $Z$-eigenvalue are introduced for studying constrained minimization problem and the necessary and sufficient conditions of such eigenvalues are  given.  It is proved that a symmetric tensor has at least one Pareto $H$-eigenvalue (Pareto $Z$-eigenvalue). Furthermore, the minimum Pareto $H$-eigenvalue (or Pareto $Z$-eigenvalue) of a symmetric tensor is exactly equal to the minimum value of constrained minimization problem of homogeneous polynomial deduced by such a tensor, which gives an alternative methods for solving the minimum value of constrained minimization problem. In particular, a symmetric tensor $\mathcal{A}$ is copositive if and only if every Pareto $H$-eigenvalue ($Z-$eigenvalue) of $\mathcal{A}$ is  non-negative.
\end{quote}

\pagestyle{fancy} \fancyhead{} \fancyhead[EC]{Yisheng Song and Liqun Qi}
\fancyhead[EL,OR]{\thepage} \fancyhead[OC]{Eigenvalue Analysis of constrained minimization problem} \fancyfoot{}

\section{\bf Introduction}\label{}
 Throughout this paper,  let $\mathbb{R}^n_{+}=\{x\in \mathbb{R}^n;x\geq0\}$, and $\mathbb{R}^n_{-}=\{x\in \mathbb{R}^n;x\leq0\}$, and $\mathbb{R}^n_{++}=\{x\in \mathbb{R}^n;x>0\}$, and $e=(1,1,\cdots,1)^T$, and $x^{[m]} = (x_1^m, x_2^m,\cdots, x_n^m)^T$ for $x = (x_1, x_2,\cdots, x_n)^T$, where $x^T$ is the transposition of a vector $x$ and $x\geq0$ ($x>0$) means $x_i\geq0$ ($x_i>0$) for all $i\in\{1,2,\cdots,n\}$.

As a natural extension of the concept of matrices, an $m$-order $n$-dimensional tensor $\mathcal{A}$ consists of $n^m$ elements in the real field $\mathbb{R}$:
$$\mathcal{A} = (a_{i_1\cdots i_m}),\ \ \ \ \  a_{i_1\cdots i_m} \in \mathbb{R},\ \   i_1,i_2,\cdots,i_m=1,2,\cdots, n.$$
For an element  $x = (x_1, x_2,\cdots, x_n)^T\in \mathbb{R}^n$ or $\mathbb{C}^n$,  $\mathcal{A}x^m$ is defined by \begin{equation}\label{eq:11}\mathcal{A}x^m=\sum_{i_1,i_2,\cdots,i_m=1}^na_{i_1i_2\cdots i_m}x_{i_1}x_{i_2}\cdots x_{i_m};\end{equation}
 $\mathcal{A}x^{m-1}$ is a vector in $\mathbb{R}^n$ (or $\mathbb{C}^n$) with its ith component defined by
\begin{equation}\label{eq:12}(\mathcal{A}x^{m-1})_i=\sum_{i_2,\cdots,i_m=1}^na_{ii_2\cdots i_m}x_{i_2}\cdots x_{i_m}\mbox{ for } i=1,2,\ldots,n.\end{equation}
An $m$-order $n$-dimensional tensor $\mathcal{A}$ is said to be {\em symmetric} if its entries $a_{i_1\cdots i_m}$ are invariant for any permutation of the indices. Clearly, each $m$-order $n$-dimensional symmetric tensor $\mathcal{A}$ defines a homogeneous polynomial $\mathcal{A}x^m$ of degree $m$ with $n$ variables and vice versa.

For given an $m$-order $n$-dimensional symmetric tensor $\mathcal{A}$, we consider a constrained optimization problem of the form:
\begin{equation}\label{eq:13} \begin{aligned}
    \min &\ \frac1m\mathcal{A}x^m\\
     s.t. &\ x^Tx^{[m-1]} = 1\\
     &\ x\in \mathbb{R}^n_+.
      \end{aligned}\end{equation}
Then the Lagrange function of the problem (\ref{eq:13}) is given clearly by
\begin{equation}\label{eq:14} L(x,\lambda,y)=\frac1m\mathcal{A}x^m+\frac1m\lambda(1-x^Tx^{[m-1]})-x^Ty\end{equation}
where $x,y\in \mathbb{R}^n_+$, $\frac{\lambda}m\in\mathbb{R}$ is the Lagrange multiplier of the equality constraint and $y$ is the Lagrange multiplier of non-negative constraint. So the solution $x$ of the problem (\ref{eq:13}) satisfies the following conditions:
\begin{align}
 \mathcal{A}x^{m-1}-\lambda x^{[m-1]}-y&=0\label{eq:15}\\
 1-x^Tx^{[m-1]} &= 0\label{eq:16}\\
x^Ty &= 0\label{eq:17}\\
x,y&\in \mathbb{R}^n_+.\label{eq:18}%\nonumber
\end{align}
The equation (\ref{eq:16}) means that $\sum\limits_{i=1}^nx_i^m = 1$. It follows from the equations (\ref{eq:15}), (\ref{eq:17}) and (\ref{eq:18}) that $$\begin{aligned}x^Ty=x^T\mathcal{A}x^{m-1}-\lambda x^Tx^{[m-1]}&=0\\ x\geq 0, \mathcal{A}x^{m-1}-\lambda x^{[m-1]}=y&\geq0,\end{aligned}$$ and hence,
\begin{equation}\label{eq:19}\begin{cases} \mathcal{A}x^m=\lambda x^Tx^{[m-1]}\\
\mathcal{A}x^{m-1}-\lambda x^{[m-1]}\geq0\\ x\geq 0.\end{cases}\end{equation}

Following Qi \cite{LQ1} ($H-$eigenvalue of the tensor $\mathcal{A}$) and Seeger \cite{S99} (Pareto eigenvalue of the matrix $A$), for  a $m$-order $n$-dimensional tensor $\mathcal{A}$, a real number $\lambda$ is called {\em Pareto $H-$eigenvalue} of the tensor $\mathcal{A}$ if there exists a non-zero vector $x\in \mathbb{R}^n$ satisfying the system (\ref{eq:19}). The non-zero vector $x$ is called a {\em Pareto $H-$eigenvector} of the tensor $\mathcal{A}$ associated to $\lambda$.

Similarly, for given an $m$-order $n$-dimensional symmetric tensor $\mathcal{A}$, we consider another constrained optimization problem of the form ($m\geq 2$):
\begin{equation}\label{eq:110} \begin{aligned}
    \min &\ \frac1m\mathcal{A}x^m\\
     s.t. &\ x^Tx = 1\\
     &\ x\in \mathbb{R}^n_+.
      \end{aligned}\end{equation}
Obviously, when $x\in \mathbb{R}^n$, $x^Tx = 1$ if and only if $(x^Tx)^{\frac{m}2} = 1$.
The corresponding Lagrange function may be written in the form
$$ L(x,\mu,y)=\frac1m\mathcal{A}x^m+\frac1m\mu(1-(x^Tx)^{\frac{m}2})-x^Ty.$$
 So the solution $x$ of the problem (\ref{eq:110})  satisfies the conditions:
$$ \mathcal{A}x^{m-1}-\mu(x^Tx)^{\frac{m}2-1} x-y=0,\
 1-(x^Tx)^{\frac{m}2}  = 0,\
x^Ty  = 0,\
x,y \in \mathbb{R}^n_+.
$$
Then we also have  $\sum\limits_{i=1}^nx_i^2 = 1$ and
\begin{equation}\label{eq:111}\begin{cases} \mathcal{A}x^m=\mu (x^Tx)^{\frac{m}2} \\
\mathcal{A}x^{m-1}-\mu (x^Tx)^{\frac{m}2-1}x \geq0\\ x\geq 0.\end{cases}\end{equation}
Following Qi \cite{LQ1} ($Z-$eigenvalue of the tensor $\mathcal{A}$) and Seeger \cite{S99} (Pareto eigenvalue of the matrix $A$), for  an $m$-order $n$-dimensional tensor $\mathcal{A}$, a real number $\mu$ is said to be {\em Pareto $Z-$eigenvalue} of the tensor $\mathcal{A}$ if there is a non-zero vector $x\in \mathbb{R}^n$ satisfying the system (\ref{eq:111}). The non-zero vector $x$ is called a {\em Pareto $Z-$eigenvector} of the tensor $\mathcal{A}$ associated to $\mu$.

So the constrained optimization problem (\ref{eq:13}) and (\ref{eq:110}) of homogeneous polynomial may be respectively solved by means of the Pareto $H$-eigenvalue (\ref{eq:19}) and Pareto $Z-$eigenvalue (\ref{eq:111}) of the corresponding tensor. It will be an interesting work to compute the Pareto $H$-eigenvalue ($Z-$eigenvalue) of a higher order tensor.\\

When $m=2$, both Pareto $H-$eigenvalue  and Pareto $Z-$eigenvalue of the $m$-order $n$-dimensional tensor obviously changes into  Pareto eigenvalue of the matrix. The concept of Pareto eigenvalue is first introduced and used by Seeger \cite{S99} for studying the equilibrium processes defined by linear complementarity conditions.  For more details, also see Hiriart-Urruty and Seeger \cite{HS10}.\\

 In this paper, we will study the properties of the Pareto $H$-eigenvalue ($Z-$eigenvalue) of a higher order tensor $\mathcal{A}$. It will be proved that a real number $\lambda$ is  Pareto $H$-eigenvalue ($Z-$eigenvalue) of  $\mathcal{A}$ if and only if $\lambda$ is $H^{++}$-eigenvalue ($Z^{++}$-eigenvalue) of some $|N|$-dimensional  principal sub-tensor of $\mathcal{A}$ with  corresponding $H-$eigenvector ($Z-$eigenvector) $w$  and
 $$\sum\limits_{i_2,\cdots ,i_m\in N}a_{ii_2\cdots i_m}w_{i_2}w_{i_3}\cdots w_{i_m}\geq0\mbox{ for }i\in \{1,2,\cdots,n\}\setminus N.$$
 So we may calculate some Pareto $H$-eigenvalue ($Z-$eigenvalue) of a higher order tensor by means of $H^{++}$-eigenvalue ($Z^{++}$-eigenvalue) of the lower dimensional tensors. What's more, we will show that \begin{equation}\label{eq:112}\min\limits_{x\geq0 \atop \|x\|_m=1 }\mathcal{A}x^m=\min\{\mu; \mu \mbox{ is Pareto $H$-eigenvalue of  }\mathcal{A}\}\end{equation}
\begin{equation}\label{eq:113}\min\limits_{x\geq0 \atop \|x\|_2=1 }\mathcal{A}x^m=\min\{\mu; \mu \mbox{ is Pareto $Z$-eigenvalue of  }\mathcal{A}\}.\end{equation} Therefore,  we may solve the constrained minimization problem for homogeneous polynomial and test the (strict) copositivity of a symmetric tensor $\mathcal{A}$ with the help of computing the Pareto $H$-eigenvalue (or Pareto $Z$-eigenvalue) of a symmetric tensor.
 As a corollary, a symmetric tensor $\mathcal{A}$ is copositive if and only if every Pareto $H$-eigenvalue ($Z-$eigenvalue) of $\mathcal{A}$ is  non-negative  and $\mathcal{A}$ is strictly copositive if and only if every Pareto $H$-eigenvalue ($Z-$eigenvalue) of $\mathcal{A}$ is   positive.

 \section{\bf Preliminaries and Basic facts}

 Let  $\mathcal{A}$ be  an $m$-order $n$-dimensional symmetric tensor. A number $\lambda\in \mathbb{C}$ is called an {\em eigenvalue of $\mathcal{A}$} if there exists a nonzero vector $x\in \mathbb{C}^n$  satisfying
\begin{equation}\label{eq:22}\mathcal{A}x^{m-1}=\lambda x^{[m-1]}, \end{equation}
where $x^{[m-1]}=(x_1^{m-1},\cdots , x_n^{m-1})^T$, and call $x$ an {\em eigenvector} of $\mathcal{A}$ associated with the eigenvalue $\lambda$. We call such an eigenvalue {\em $H$-eigenvalue} if it is real and has a real eigenvector $x$, and call such a real eigenvector $x$ an {\em H-eigenvector}.

These concepts were first introduced by Qi \cite{LQ1} to the higher order symmetric tensor, and the existence of the eigenvalues and its some application were studied also. Lim \cite{LL} independently introduced these concept and obtained the existence results using the variational approach. Qi \cite{LQ1, LQ2, LQ3} extended some nice properties of the matrices to the higher order tensors.  Subsequently,
this topics are attracted attention of many mathematicians
from different disciplines. For various studies and applications, see  Chang \cite{C09}, Chang, Pearson and Zhang \cite{CPT1},
Chang, Pearson and Zhang \cite{CPT},  Hu, Huang and Qi \cite{HHQ}, Hu and Qi \cite{HQ}, Ni, Qi, Wang and Wang \cite{NQWW}, Ng, Qi and Zhou \cite{NQZ},  Song and Qi \cite{SQ13,SQ10}, Yang and Yang \cite{YY10,YY11}, Zhang \cite{TZ}, Zhang and  Qi \cite{ZQ}, Zhang, Qi and Xu \cite{ZQX} and references cited therein.\\

A number $\mu\in \mathbb{C}$ is said to be an {\em $E$-eigenvalue of $\mathcal{A}$} if there exists a nonzero
vector $x\in \mathbb{C}^n$ such that
\begin{equation}\label{eq:23}\mathcal{A}x^{m-1}=\mu x (x^Tx)^{\frac{m-2}2}.\end{equation} Such a nonzero
vector $x\in \mathbb{C}^n$ is called an {\em $E$-eigenvector} of $\mathcal{A}$ associated with $\mu$,
 If $x$ is real, then $\mu$ is also real. In this case, $\mu$ and $x$ are called a {\em $Z$-eigenvalue} of  $\mathcal{A}$ and
a {\em $Z$-eigenvector} of  $\mathcal{A}$ (associated with $\mu$), respectively. Qi \cite{LQ1, LQ2, LQ3} first introduced and used these concepts  and showed that if $\mathcal{A}$ is regular, then  a complex number is an $E$-eigenvalue of higher order symmetric tensor if and only if it is a root of the corresponding $E$-characteristic polynomial. Also see Hu and Qi \cite{HQ13},  Hu,  Huang,  Ling and  Qi \cite{HHLQ}, Li, Qi and Zhang \cite{LQZ} for more details.  \\

In homogeneous
polynomial $\mathcal{A}x^m$ defined by (\ref{eq:11}), if we let some (but not all) $x_i$ be zero, then we have a homogeneous
polynomial with fewer variables, which defines a lower dimensional tensor. We call such a lower dimensional
 tensor a {\em principal sub-tensor} of $\mathcal{A}$. The concept were first introduced and used by Qi \cite{LQ1} to the higher order symmetric tensor.\\

Recently, Qi \cite{LQ4} introduced and used the following concepts for studying the properties of hypergraph. An $H$-eigenvalue $\lambda$ of $\mathcal{A}$ is said to be (i) {\em  $H^+$-eigenvalue of $\mathcal{A}$}, if its $H$-eigenvector $x\in \mathbb{R}^n_+$; (ii) {\em  $H^{++}$-eigenvalue of $\mathcal{A}$}, if its $H$-eigenvector $x\in \mathbb{R}^n_{++}$. Similarly, we introduce the concepts of $Z^+$-eigenvalue and $Z^{++}$-eigenvalue.  An $Z$-eigenvalue $\mu$ of $\mathcal{A}$ is said to be (i) {\em $Z^+$-eigenvalue of $\mathcal{A}$}, if its $Z$-eigenvector $x\in \mathbb{R}^n_+$; (ii) {\em  $Z^{++}$-eigenvalue of $\mathcal{A}$}, if its $Z$-eigenvector $x\in \mathbb{R}^n_{++}$.

\section{\bf Pareto $H$-eigenvalue and Pareto $Z$-eigenvalue}

Let $N$ be  a subset of the index
set $\{1, 2,  \cdots, n\}$ and $\mathcal{A}$ be a tensor of order $m$ and dimension $n$.  We denote the principal sub-tensor of   $\mathcal{A}$ by $\mathcal{A} ^N$ which is obtained by homogeneous polynomial $\mathcal{A}x^m$ for all $x=(x_1, x_2, \cdots, x_n)^T$ with $x_i=0$ for $i\in \{1, 2,  \cdots, n\}\setminus N$.   The symbol $|N|$ denotes the cardinality of $N$.  So, $\mathcal{A} ^N$ is a tensor of order $m$ and dimension $|N|$ and the principal sub-tensor $\mathcal{A}^N$  is just $\mathcal{A}$ itself when $N=\{1, 2,  \cdots, n\}$.

\begin{thm} \label{th:31} Let $\mathcal{A}$ be a  $m$-order and $n$-dimensional  tensor.
A real number $\lambda$ is  Pareto $H$-eigenvalue of  $\mathcal{A}$ if and only if there exists a nonempty subset $N\subseteq \{1, 2,  \cdots, n\}$ and a vector $w\in\mathbb{R}^{|N|}$
 such that
\begin{align}
 \mathcal{A}^{N}w^{m-1}&=\lambda w^{[m-1]}\label{eq:31},\  \  w\in\mathbb{R}^{|N|}_{++}\\
 \sum\limits_{i_2,\cdots ,i_m\in N}a_{ii_2\cdots i_m}w_{i_2}w_{i_3}\cdots w_{i_m}&\geq0\mbox{ for }i\in \{1,2,\cdots,n\}\setminus N\label{eq:32}%\nonumber
\end{align}
In such a case, the vector $y\in\mathbb{R}^{|N|}_+$ defined by
\begin{equation}\label{eq:33} y_i=\begin{cases} w_i, \ \  i\in N\\
0,\ \ i\in \{1,2,\cdots,n\}\setminus N\end{cases}\end{equation}
is a Pareto $H$-eigenvector of $\mathcal{A}$ associated to the real number $\lambda$.
 \end{thm}

\begin{proof} First we show the necessity.  Let the real number $\lambda$ be a Pareto $H$-eigenvalue of $\mathcal{A}$ with a corresponding Pareto $H$-eigenvector $y$. Then by the definition (\ref{eq:19}) of the Pareto $H$-eigenvalue, the Pareto $H$-eigenpairs $(\lambda,y)$  may be rewritten in the form
 \begin{equation}\label{eq:34}\begin{aligned} y^T(\mathcal{A}y^{m-1}-\lambda y^{[m-1]})=&0\\
\mathcal{A}y^{m-1}-\lambda y^{[m-1]}\geq&0\\ y\geq& 0\end{aligned}\end{equation}
and hence \begin{align} \sum_{i=1}^ny_i(\mathcal{A}y^{m-1}-\lambda y^{[m-1]})_i=&0\label{eq:35}\\
(\mathcal{A}y^{m-1}-\lambda y^{[m-1]})_i\geq&0,\ \mbox{ for } i=1,2,\ldots,n\label{eq:36}\\
y_i\geq& 0,\ \mbox{ for } i=1,2,\ldots,n.\label{eq:37}%\nonumber
\end{align}
Combining the equation (\ref{eq:35}) with (\ref{eq:36}) and (\ref{eq:37}), we have
\begin{equation}\label{eq:38}y_i(\mathcal{A}y^{m-1}-\lambda y^{[m-1]})_i=0,\ \mbox{ for all } i\in \{1,2,\ldots,n\}.\end{equation}
 Take $N=\{i\in \{1,2,\ldots,n\}; y_i>0\}$. Let the vector $w\in\mathbb{R}^{|N|}$ be defined by $$w_i=y_i\mbox{ for all } i\in N.$$
Clearly, $w\in\mathbb{R}^{|N|}_{++}$.  Combining the equation (\ref{eq:38}) with the fact that $y_i>0$ for all $i\in N$, we have
 $$(\mathcal{A}y^{m-1}-\lambda y^{[m-1]})_i=0,\ \mbox{ for all } i\in N,$$ and so
 $$\mathcal{A}^{N}w^{m-1}=\lambda w^{[m-1]},\  \  w\in\mathbb{R}^{|N|}_{++}.$$
 It follows from the equation (\ref{eq:36}) and the fact that $y_i=0$ for all $i\in \{1,2,\cdots,n\}\setminus N$ that $$(\mathcal{A}y^{m-1})_i\geq0,\ \mbox{ for all } i\in \{1,2,\cdots,n\}\setminus N.$$
 By the definition (\ref{eq:12}) of $\mathcal{A}y^{m-1}$, the conclusion (\ref{eq:32}) holds.

 Now we show the sufficiency. Suppose that there exists a nonempty subset $N\subseteq \{1, 2,  \cdots, n\}$ and a vector $w\in\mathbb{R}^{|N|}$
 satisfying (\ref{eq:31}) and (\ref{eq:32}). Then the vector $y$ defined by (\ref{eq:33}) is a non-zero vector in $\mathbb{R}^{|N|}_+$ such that $(\lambda,y)$ satisfying (\ref{eq:34}). The desired conclusion follows.
 \end{proof}

 Using the same proof techniques as that of Theorem \ref{th:31} with appropriate changes in the
inequalities or equalities ($y^{[m-1]}$ is replaced by $(y^Ty)^{\frac{m-2}2} y$ and so on). We can obtain the following conclusions about the Pareto $Z$-eigenvalue of $\mathcal{A}$.

 \begin{thm} \label{th:32} Let $\mathcal{A}$ be a  $m$-order and $n$-dimensional  tensor.
A real number $\mu$ is  Pareto $Z$-eigenvalue of  $\mathcal{A}$ if and only if there exists a nonempty subset $N\subseteq \{1, 2,  \cdots, n\}$ and a vector $w\in\mathbb{R}^{|N|}$
 such that
\begin{align}
 \mathcal{A}^{N}w^{m-1}&=\mu (w^Tw)^{\frac{m-2}2} w\label{eq:39},\  \  w\in\mathbb{R}^{|N|}_{++}\\
 \sum\limits_{i_2,\cdots ,i_m\in N}a_{ii_2\cdots i_m}w_{i_2}w_{i_3}\cdots w_{i_m}&\geq0\mbox{ for }i\in \{1,2,\cdots,n\}\setminus N\label{eq:310}%\nonumber
\end{align}
In such a case, the vector $y\in\mathbb{R}^{|N|}_+$ defined by
\begin{equation}\label{eq:311} y_i=\begin{cases} w_i, \ \  i\in N\\
0,\ \ i\in \{1,2,\cdots,n\}\setminus N\end{cases}\end{equation}
is a Pareto $Z$-eigenvector of $\mathcal{A}$ associated to the real number $\mu$.
 \end{thm}

Following Theroem \ref{th:31} and \ref{th:32}, the following results are obvious.

 \begin{cor} \label{co:33} Let $\mathcal{A}$ be a $m$-order and $n$-dimensional  tensor.
If a real number $\lambda$ is  Pareto $H$-eigenvalue ($Z$-eigenvalue) of  $\mathcal{A}$, then $\lambda$ is $H^{++}$-eigenvalue ($Z^{++}$-eigenvalue, respectively) of some $|N|$-dimensional  principal sub-tensor of $\mathcal{A}$.
 \end{cor}

 Since the definition of $H^{+}$-eigenvalue ($Z^{+}$-eigenvalue) $\lambda$ of $\mathcal{A}$ means that $\mathcal{A}x^{m-1}-\lambda x^{[m-1]}=0$ ($\mathcal{A}x^{m-1}-\lambda (x^Tx)^{\frac{m}2-1} x=0$, respectively) for some non-zero vector $x\geq0$, the following conclusions are trivial.

 \begin{prop} \label{pr:34} Let $\mathcal{A}$ be a $m$-order and $n$-dimensional  tensor. Then
 \begin{itemize}
\item[(i)] each $H^{+}$-eigenvalue ($Z^{+}$-eigenvalue) of $\mathcal{A}$ is its Pareto $H$-eigenvalue ($Z$-eigenvalue, respectively);
\item[(ii)]the Pareto $H$-eigenvalues ($Z$-eigenvalues) of a diagonal tensor $\mathcal{A}$ coincide with its diagonal
entries. In particular, a $n$-dimensional and diagonal tensor  may have at most
$n$ distinct Pareto $H$-eigenvalues ($Z$-eigenvalues).
\end{itemize}
 \end{prop}

 It follows from the above results that  some Pareto $H$-eigenvalue ($Z-$eigenvalue) of a higher order tensor may be calculated by means of $H^{++}$-eigenvalue ($Z^{++}$-eigenvalue, respectively) of the lower dimensional tensors.

\begin{ex} \label{ex:1} Let $\mathcal{A}$ be a $4$-order and $2$-dimensional tensor. Suppose that $a_{1111}=1, a_{2222}=2$, $a_{1122}+a_{1212}+a_{1221}=-1$, $a_{2121}+a_{2112}+a_{2211}=-2$, and other $a_{i_1i_2i_3i_4}=0$. Then $$\mathcal{A}x^4=x_1^4+2x_2^4-3x_1^2x_2^2$$ $$\mathcal{A}x^3=\left(\begin{aligned}x_1^3-&x_1x_2^2\\2x_2^3-&2x_1^2x_2\end{aligned}\right)$$

 When $N=\{1, 2\}$, the principal sub-tensor $\mathcal{A}^N$  is just $\mathcal{A}$ itself. $\lambda_1=0$ is a $H^{++}$-eigenvalue of $\mathcal{A}$ with a corresponding eigenvector $x^{(1)}=(\frac{\sqrt[4]{8}}2,\frac{\sqrt[4]{8}}2)^T$, and so it follows from Theorem \ref{th:31} that $\lambda_1=0$ is a Pareto $H$-eigenvalue with Pareto $H$-eigenvector $x^{(1)}=(\frac{\sqrt[4]{8}}2,\frac{\sqrt[4]{8}}2)^T$.

  $\lambda_2=0$ is a $Z^{++}$-eigenvalue of $\mathcal{A}$ with a corresponding eigenvector $x^{(2)}=(\frac{\sqrt2}2,\frac{\sqrt2}2)^T$, and so it follows from Theorem \ref{th:32} that $\lambda_2=0$ is a Pareto $Z$-eigenvalue of $\mathcal{A}$ with Pareto $Z$-eigenvector $x^{(2)}=(\frac{\sqrt2}2,\frac{\sqrt2}2)^T$.

  When $N=\{1\}$, the $1$-dimensional principal sub-tensor $\mathcal{A}^{N}=1$. Obviously, $\lambda_3=1$ is both $H^{++}$-eigenvalue and $Z^{++}$-eigenvalue of $\mathcal{A}^{N}$ with a corresponding eigenvector $w=1$ and $a_{2111}w^3=0$,  and hence it follows from Theorem \ref{th:31} and \ref{th:32} that $\lambda_3=1$ is both Pareto $H$-eigenvalue and Pareto $Z$-eigenvalue of $\mathcal{A}$ with a corresponding eigenvector $x^{(3)}=(1,0)^T$.

   Similarly, when $N=\{2\}$, the $1$-dimensional principal sub-tensor $\mathcal{A}^{N}=2$. Clearly, $\lambda_4=2$ is both $H^{++}$-eigenvalue and $Z^{++}$-eigenvalue of $\mathcal{A}^{N}$ with a corresponding eigenvector $w=1$ and $a_{1222}w^3=0$,  and so $\lambda_4=2$ is both Pareto $H$-eigenvalue and Pareto $Z$-eigenvalue of $\mathcal{A}$ with a corresponding eigenvector $x^{(4)}=(0,1)^T$.
\end{ex}

 \begin{ex} \label{ex:2} Let $\mathcal{A}$ be a $3$-order and $2$-dimensional tensor. Suppose that $a_{111}=1, a_{222}=2$, $a_{122}=a_{212}=a_{221}=\frac13$, and $a_{112}=a_{121}=a_{211}=-\frac23$. Then $$\mathcal{A}x^3=x_1^3+x_1x_2^2-2x_1^2x_2+2x_2^3$$ $$\mathcal{A}x^2=\left(\begin{aligned}x_1^2&+\frac13x_2^2-\frac43x_1x_2\\2x_2^2&+\frac23x_1x_2-\frac23x_1^2\end{aligned}\right)$$

  When $N=\{1\}$, the $1$-dimensional principal sub-tensor $\mathcal{A}^{N}=1$. Obviously, $\lambda_1=1$ is both $H^{++}$-eigenvalue and $Z^{++}$-eigenvalue of $\mathcal{A}^{N}$ with a corresponding eigenvector $w=1$ and $a_{211}w^2=-\frac23<0$, and so $\lambda_1=1$ is neither Pareto $H$-eigenvalue nor Pareto $Z$-eigenvalue of $\mathcal{A}$.

When $N=\{2\}$, the $1$-dimensional principal sub-tensor $\mathcal{A}^{N}=2$. Clearly, $\lambda_2=2$ is both $H^{++}$-eigenvalue and $Z^{++}$-eigenvalue of $\mathcal{A}^{N}$ with a corresponding eigenvector $w=1$ and $a_{122}w^2=\frac13>0$, and so $\lambda_2=2$ is both Pareto $H$-eigenvalue and Pareto $Z$-eigenvalue of $\mathcal{A}$ with a corresponding eigenvector $x^{(2)}=(0,1)^T$. But $\lambda=2$ is neither  $H^+$-eigenvalue nor $Z^+$-eigenvalue of $\mathcal{A}$.
\end{ex}

 \begin{rem}
 The Example \ref{ex:2} reveals that a Pareto $H$-eigenvalue ($Z$-eigenvalue) of a tensor $\mathcal{A}$ may not be  its $H^+$-eigenvalue ($Z^{+}$-eigenvalue) even when $\mathcal{A}$ is symmetric.
 \end{rem}
\section{\bf Constrained minimization and Pareto eigenvalue}

Let $\mathcal{A}$ be a symmetric tensor of order $m$ and dimension $n$ and $\|x\|_k=(|x_1|^k+|x_2|^k+\cdots+|x_n|^k)^{\frac1k} $ for $k\geq1$.  Denote by $e^{(i)}=(e^{(i)}_1,e^{(i)}_2,\cdots,e^{(i)}_n)^T$ the ith unit vector in $\mathbb{R}^n$, i.e.,
 $$e^{(i)}_j=\begin{cases}1 &\mbox{ if }i=j\\ 0  & \mbox{ if }i\neq j\end{cases}\mbox{ for }i,j\in\{1,2,\cdots,n\}.$$
We consider the constrained minimization problem  \begin{equation}
 \gamma(\mathcal{A})=\min\{\mathcal{A}x^m;\  x\geq0\mbox{ and }\|x\|_m=1 \},\label{eq:41} \end{equation}

\begin{thm} \label{th:41} Let $\mathcal{A}$ be a  $m$-order and $n$-dimensional  symmetric tensor. If  $$\lambda(\mathcal{A})=\min\{\lambda; \lambda \mbox{ is Pareto $H$-eigenvalue of  }\mathcal{A}\},$$ then $\gamma(\mathcal{A})=\lambda(\mathcal{A})$.
 \end{thm}
 \begin{proof}
 Let $\lambda$ be a Pareto $H$-eigenvalue of  $\mathcal{A}$.  Then there exists a non-zero vector $y\in\mathbb{R}^n$ such that
 $$ \mathcal{A}y^m=\lambda y^Ty^{[m-1]},\ y\geq0,$$
 and so
 \begin{equation} \mathcal{A}y^m=\lambda \sum_{i=1}^ny_i^m=\lambda \|y\|_m^m\mbox{ and }\|y\|_m>0.\label{eq:42}\end{equation}
 Then we have $$\lambda=\mathcal{A}(\frac{y}{\|y\|_m})^m\mbox{ and } \|\frac{y}{\|y\|_m}\|_m=1.$$
 From (\ref{eq:41}), it follows that $\gamma(\mathcal{A})  \leq\lambda.$ Since $\lambda$ is arbitrary,
 we have $$\gamma(\mathcal{A})\leq\lambda(\mathcal{A}).$$

 Now we show $\gamma(\mathcal{A})\geq\lambda(\mathcal{A}).$ Let $S=\{x\in\mathbb{R}^n; x\geq0\mbox{ and }\|x\|_m=1\}.$  It follows from the continuity of the homogeneous polynomial $\mathcal{A}x^m$ and the compactness of the set $S$ that there exists a $v\in S$ such that
  \begin{equation}\gamma(\mathcal{A})=\mathcal{A}v^m,\ v\geq0,\ \|v\|_m=1.\label{eq:43}\end{equation}
 Let $g(x)=\mathcal{A}x^m-\gamma(\mathcal{A})x^Tx^{[m-1]}$ for all $x\in\mathbb{R}^n$. We claim that for all $x\geq0,$ $g(x)\geq0.$ Suppose not, then there exists non-zero vector $y\geq0$ such that $$g(y)=\mathcal{A}y^m-\gamma(\mathcal{A})\sum_{i=1}^ny_i^m<0,$$ and hence $\gamma(\mathcal{A})\leq\mathcal{A}(\frac{y}{\|y\|_m})^m<\gamma(\mathcal{A}), $ a contradiction.  Thus we have
 \begin{equation} g(x)=\mathcal{A}x^m-\gamma(\mathcal{A})x^Tx^{[m-1]}\geq0\mbox{ for all }x\in \mathbb{R}^n_+.\label{eq:44}\end{equation}

For each $i\in\{1,2,\cdots,n\}$,  we define a one-variable function $$f(t)=g(v+t e^{(i)})\mbox{ for all }t\in\mathbb{R}^1.$$  Clearly, $f(t)$ is continuous and $v+t e^{(i)}\in \mathbb{R}^n_+$  for all $t\geq 0.$  It follows from (\ref{eq:43}) and  (\ref{eq:44}) that $$f(0)=g(v)=0 \mbox{ and } f(t)\geq0\mbox{ for all }t\geq0.$$
 From the necessary conditions of  extremum of one-variable function, it follows that the right-hand derivative $f'_+(0)\geq0$, and hence
 $$\begin{aligned}f'_+(0)=(e^{(i)})^T\nabla g(v) =&m(e^{(i)})^T(\mathcal{A}v^{m-1}-\gamma(\mathcal{A}) v^{[m-1]})\\
 =&m(\mathcal{A}v^{m-1}-\gamma(\mathcal{A}) v^{[m-1]})_i\geq0.\end{aligned}$$  So we have $$(\mathcal{A}v^{m-1}-\gamma(\mathcal{A}) v^{[m-1]})_i\geq0, \mbox{ for } i\in\{1,2,\cdots,n\}.$$
 Therefore, we obtain
 \begin{align} f(0)=g(v)=\mathcal{A}v^m-\gamma(\mathcal{A}) v^Tv^{[m-1]}=&0\label{eq:45} \\
\mathcal{A}v^{m-1}-\gamma(\mathcal{A}) v^{[m-1]}\geq&0\label{eq:46}\\v\geq&0\nonumber
\end{align}
 Namely,  $\gamma(\mathcal{A})$ is a Pareto $H$-eigenvalue of  $\mathcal{A}$, and hence $\gamma(\mathcal{A})\geq\lambda(\mathcal{A}),$  as required. \end{proof}

 It follows from the proof of the inquality $\gamma(\mathcal{A})\geq\lambda(\mathcal{A})$ in Theorem \ref{th:41} that $\gamma(\mathcal{A})$ is a Pareto $H$-eigenvalue of  $\mathcal{A}$, which implies the existence of Pareto $H$-eigenvalue of a symmetric tensor $\mathcal{A}$.

 \begin{thm} \label{th:42} If a $m$-order and $n$-dimensional tensor $\mathcal{A}$ is symmetric,  then $\mathcal{A}$ has at least one Pareto $H$-eigenvalue $\gamma(\mathcal{A})=\min\limits_{ x\geq0 \atop \|x\|_m=1 }\mathcal{A}x^m$.
 \end{thm}

Since $(x^Tx)^{\frac{m}2}=\|x\|_2^m$, using the same proof techniques as that of Theorem \ref{th:41} with appropriate changes in the
inequalities or equalities ($x^Tx^{[m-1]}$ and $y^{[m-1]}$ are respectively replaced by $(x^Tx)^{\frac{m}2}$ and $(y^Ty)^{\frac{m-2}2} y$). We can obtain the following conclusions about the Pareto $Z$-eigenvalue of  a symmetric tensor $\mathcal{A}$.

\begin{thm} \label{th:43} Let $\mathcal{A}$ be a  $m$-order and $n$-dimensional  symmetric tensor. Then $\mathcal{A}$ has at least one Pareto $Z$-eigenvalue $\mu(\mathcal{A})=\min\limits_{x\geq0 \atop \|x\|_2=1 }\mathcal{A}x^m$. What's more,
\begin{equation}\label{eq:47}\mu(\mathcal{A})=\min\{\mu; \mu \mbox{ is Pareto $Z$-eigenvalue of  }\mathcal{A}\}.\end{equation}
 \end{thm}

 In 1952, Motzkin \cite{TSM} introduced the concept of copositive matrices, which is an important in applied mathematics and graph theory. A  real symmetric matrix $A$ is said to
be (i) {\em copositive} if $x\geq 0$ implies $x^TAx\geq0$; (ii) {\em strictly copositive} if $x\geq 0$ and $x\neq0$ implies $x^TAx>0$. Recently, Qi \cite{LQ5} extended this concept to the higher order symmetric tensors and obtained its some nice properties as ones of copositive matrices.
Let $\mathcal{A}$ be a real symmetric tensor of order $m$ and dimension $n$. $\mathcal{A}$ is said to be \begin{itemize}
\item[(i)] {\em copositive } if  $\mathcal{A}x^m\geq0$ for all $x\in \mathbb{R}^n_+$; \item[(ii)] {\em strictly copositive} if $\mathcal{A}x^m>0$ for all $x\in \mathbb{R}^n_+\setminus\{0\}$.\end{itemize}

   Let $\|\cdot\|$ denote any norm on $\mathbb{R}^n$. Obviously, we have the following equivalent definition of  (strict) copositivity of a symmetric tensor in the sense of any norm on $\mathbb{R}^n$. Also see Song and Qi \cite{SQ} for detail proof.

\begin{lem}(Song and Qi \cite{SQ}) \label{le:44} Let $\mathcal{A}$ be a symmetric tensor of order $m$ and dimension $n$. Then we have
\begin{itemize}
\item[(i)] $\mathcal{A}$ is copositive if and only if  $\mathcal{A}x^m\geq0$ for all $x\in \mathbb{R}^n_+$ with $\|x\|=1$;
\item[(ii)] $\mathcal{A}$ is strictly copositive if and only if $\mathcal{A}x^m>0$ for all $x\in \mathbb{R}^n_+$ with $\|x\|=1$;
\end{itemize}
 \end{lem}

 %\begin{proof}
 %(i) When $\mathcal{A}$ is copositive, the conclusion is obvious. Conversely, take $x\in \mathbb{R}^n_+$. If $\|x\|=0$, then it follows that $x=0$, and hence $\mathcal{A}x^m=0$. If $\|x\|>0$, then let $y=\frac{x}{\|x\|}$. We have $\|y\|=1$ and $x=\|x\|y$, and so
 %$$\mathcal{A}x^m=\mathcal{A}(\|x\|y)^m=\|x\|^m\mathcal{A}y^m\geq0.$$
 %Therefore, $\mathcal{A}x^m\geq0$ for all $x\in \mathbb{R}^n_+$, as required.

 %Similarly, (ii) is easily proved.
 %\end{proof}

 As the immediate conclusions of the above consequences, it is easy to obtain the following results about the copositive  (strictly copositive) tensor.

 \begin{cor}\label{co:45} Let $\mathcal{A}$ be a $m$-order and $n$-dimensional symmetric tensor. Then \begin{itemize}
\item[(a)] $\mathcal{A}$ always has Pareto $H$-eigenvalue. $\mathcal{A}$ is copositive  (strictly copositive) if and only
if all of its Pareto $H$-eigenvalues are nonnegative (positive, respectively).
\item[(b)] $\mathcal{A}$ always has Pareto $Z$-eigenvalue. $\mathcal{A}$ is copositive  (strictly copositive) if and only
if all of its Pareto $Z$-eigenvalues are nonnegative (positive, respectively).\end{itemize}
\end{cor}

Now we give an example for solving the constrained minimization problem for homogeneous polynomial and testing the (strict) copositivity of a symmetric tensor $\mathcal{A}$ with the help of the above results.

 \begin{ex} \label{ex:3} Let $\mathcal{A}$ be a $4$-order and $2$-dimensional tensor. Suppose that $a_{1111}=a_{2222}=1$, $a_{1112}=a_{1211}=a_{1121}=a_{2111}=t$, and other $a_{i_1i_2i_3i_4}=0$. Then $$\mathcal{A}x^4=x_1^4+x_2^4+4tx_1^3x_2$$ $$\mathcal{A}x^3=\left(\begin{aligned}x_1^3+&3tx_1^2x_2\\x_2^3+&tx_1^3\end{aligned}\right)$$

  When $N=\{1, 2\}$, the principal sub-tensor $\mathcal{A}^N$  is just $\mathcal{A}$ itself. $\lambda_1=1+\sqrt[4]{27}t$  is $H^{++}$-eigenvalue of $\mathcal{A}$ with a corresponding eigenvector $x^{(1)}=(\sqrt[4]{\frac34},\sqrt[4]{\frac14})^T$. Then it follows from Theorem \ref{th:31} and \ref{th:32} that $\lambda_1=1+\sqrt[4]{27}t$  is  Pareto $H$-eigenvalues with Pareto $H$-eigenvector $x^{(1)}=(\sqrt[4]{\frac34},\sqrt[4]{\frac14})^T$.

    When $N=\{1\}$, the $1$-dimensional principal sub-tensor $\mathcal{A}^{N}=1$. Obviously, $\lambda_2=1$ is both $H^{++}$-eigenvalue and $Z^{++}$-eigenvalue of $\mathcal{A}^{N}$ with a corresponding eigenvector $w=1$ and $a_{2111}w^3=t$.   Then when $t>0$,  it follows from Theorem \ref{th:31} and \ref{th:32} that $\lambda_2=1$ is both Pareto $H$-eigenvalue and Pareto $Z$-eigenvalue of $\mathcal{A}$ with a corresponding eigenvector $x^{(2)}=(1,0)^T$; when $t<0$, $\lambda_2=1$ is neither Pareto $H$-eigenvalue nor Pareto $Z$-eigenvalue of $\mathcal{A}$.

   Similarly, when $N=\{2\}$, the $1$-dimensional principal sub-tensor $\mathcal{A}^{N}=1$. Clearly, $\lambda_3=1$ is both $H^{++}$-eigenvalue and $Z^{++}$-eigenvalue of $\mathcal{A}^{N}$ with a corresponding eigenvector $w=1$ and $a_{1222}w^3=0$,  and so $\lambda_3=1$ is both Pareto $H$-eigenvalue and Pareto $Z$-eigenvalue of $\mathcal{A}$ with a corresponding eigenvector $x^{(3)}=(0,1)^T$.

   So  the following conclusions are easily obtained:
\begin{itemize}
\item[(i)]  Let $t<-\frac{1}{\sqrt[4]{27}}$. Then $\lambda_1=1+\sqrt[4]{27}t<0$ and  $\lambda_3=1$ are  Pareto $H$-eigenvalues of $\mathcal{A}$ with Pareto $H$-eigenvectors $x^{(1)}=(\sqrt[4]{\frac34},\sqrt[4]{\frac14})^T$ and $x^{(3)}=(0,1)^T$, respectively.  It follows from Theorem \ref{th:41} and \ref{th:42} that $$\gamma(\mathcal{A})=\min\limits_{x\geq0 \atop \|x\|_4=1}\mathcal{A}x^4=\min\{\lambda_1,\lambda_3\}=1+\sqrt[4]{27}t<0.$$ The polynomial $\mathcal{A}x^4$ attains its minimum value at $x^{(1)}=(\sqrt[4]{\frac34},\sqrt[4]{\frac14})^T$.  It follows from Corollary \ref{co:45} that  $\mathcal{A}$ is not copositive.
\item[(ii)]  Let $t=-\frac{1}{\sqrt[4]{27}}$. Then $\lambda_1=1+\sqrt[4]{27}t=0$ and  $\lambda_3=1$ are  Pareto $H$-eigenvalues of $\mathcal{A}$ with Pareto $H$-eigenvectors $x^{(1)}=(\sqrt[4]{\frac34},\sqrt[4]{\frac14})^T$ and $x^{(3)}=(0,1)^T$, respectively.  It follows from Theorem \ref{th:41} and \ref{th:42} that $$\gamma(\mathcal{A})=\min\limits_{x\geq0 \atop \|x\|_4=1}\mathcal{A}x^4=\min\{\lambda_1,\lambda_3\}=0.$$ The polynomial $\mathcal{A}x^4$ attains its minimum value at $x^{(1)}=(\sqrt[4]{\frac34},\sqrt[4]{\frac14})^T$.  It follows from Corollary \ref{co:45} that  $\mathcal{A}$ is copositive.
\item[(iii)] Let $0>t>-\frac{1}{\sqrt[4]{27}}$. Clearly, $0<1+\sqrt[4]{27}t<1$. Then $\lambda_1=1+\sqrt[4]{27}t$ and  $\lambda_3=1$ are  Pareto $H$-eigenvalues of $\mathcal{A}$.  It follows from Theorem \ref{th:41} and \ref{th:42} that $$\gamma(\mathcal{A})=\min\limits_{x\geq0 \atop \|x\|_4=1}\mathcal{A}x^4=\min\{\lambda_1,\lambda_3\}=1+\sqrt[4]{27}t>0.$$ The polynomial $\mathcal{A}x^4$ attains its minimum value at  $x^{(1)}=(\sqrt[4]{\frac34},\sqrt[4]{\frac14})^T$. It follows from Corollary \ref{co:45} that  $\mathcal{A}$ is strictly copositive.
\item[(iv)] Let $t=0$.  Then $\lambda_1=\lambda_2=\lambda_3=1$ are Pareto $H$-eigenvalues of $\mathcal{A}$  with Pareto $H$-eigenvectors $x^{(1)}=(\sqrt[4]{\frac34},\sqrt[4]{\frac14})^T$ and $x^{(2)}=(1,0)^T$ and $x^{(3)}=(0,1)^T$, respectively.  It follows from Theorem \ref{th:41} and \ref{th:42} that $$\gamma(\mathcal{A})=\min\limits_{x\geq0 \atop \|x\|_4=1}\mathcal{A}x^4=\min\{\lambda_1,\lambda_2,\lambda_3\}=1>0.$$
The polynomial $\mathcal{A}x^4$ attains its minimum value at $x^{(1)}=(\sqrt[4]{\frac34},\sqrt[4]{\frac14})^T$ or $x^{(2)}=(1,0)^T$ or $x^{(3)}=(0,1)^T$.  It follows from Corollary \ref{co:45} that  $\mathcal{A}$ is strictly copositive.
\item[(v)] Let $t>0$.  Then $\lambda_1=1+\sqrt[4]{27}t$ and  $\lambda_2=\lambda_3=1$ are  Pareto $H$-eigenvalues of $\mathcal{A}$  with Pareto $H$-eigenvectors $x^{(1)}=(\sqrt[4]{\frac34},\sqrt[4]{\frac14})^T$ and $x^{(2)}=(1,0)^T$ and $x^{(3)}=(0,1)^T$, respectively.  It follows from Theorem \ref{th:41} and \ref{th:42} that $$\gamma(\mathcal{A})=\min\limits_{x\geq0 \atop \|x\|_4=1}\mathcal{A}x^4=\min\{\lambda_1,\lambda_2,\lambda_3\}=1>0.$$
The polynomial $\mathcal{A}x^4$ attains its minimum value at $x^{(2)}=(1,0)^T$ or $x^{(3)}=(0,1)^T$.  It follows from Corollary \ref{co:45} that  $\mathcal{A}$ is strictly copositive.
 \end{itemize}
\end{ex}

%\section{Acknowledgments}

% ----------------------------------------------------------------
\bibliographystyle{amsplain}

\end{document}